\documentclass[11pt]{article}
\usepackage{amscd}
\usepackage{color}
\usepackage{amsfonts}
\usepackage[all]{xy}
\usepackage{amsmath,amssymb,amsthm,latexsym,mathdots}
\usepackage{amscd}

\usepackage{bm}
\usepackage{setspace}
\usepackage{caption}

\newtheorem{theorem}{Theorem}[section]
\newtheorem{proposition}[theorem]{Proposition}
\newtheorem{definition}[theorem]{Definition}
\newtheorem{corollary}[theorem]{Corollary}
\newtheorem{lemma}[theorem]{Lemma}
\numberwithin{equation}{section}
\theoremstyle{remark}
\newtheorem{remark}[theorem]{Remark}
\newtheorem{example}[theorem]{\bf Example}
\newcommand{\R}{\mathbb{R}}

\begin{document}
\title{\bf{A characterization of the Ejiri torus in $S^{5}$}}
\author{Peng Wang}

\date{}
\maketitle

\begin{center}
{\bf Abstract}
\end{center}

Ejiri's torus in $S^5$ is the first example of Willmore surface which is not conformally equivalent to any minimal surface in any space forms. Li and Vrancken classified all Willmore surfaces of tensor product in $S^{n}$ by reducing them into elastic curves in $S^3$, and the Ejiri torus appeared as a special example. In this paper, we first prove that among all Willmore tori of tensor product, the Willmore functional of the Ejiri torus in $S^5$ attains the minimum $2\pi^2\sqrt{3}$. Then we show that all Willmore tori of tensor product are unstable when the co-dimension is big enough. We also show that the Ejiri torus is unstable in $S^5$.
Moreover, similar to Li and Vrancken, we classify all constrained Willmore surfaces of tensor product by reducing them with  elastic curves in $S^3$. All constrained Willmore tori obtained this way are also shown to be unstable when the co-dimension is big enough.

We conjecture that a Willmore torus having Willmore functional between $2\pi^2$ and $2\pi^2\sqrt{3}$ is either the Clifford torus, or the Ejiri torus (up to some conformal transform).
\\


{\bf MSC(2000):\hspace{2mm} 53A30, 53C15}

\section{Introduction}

For a surface $x:M\rightarrow S^n$ with  mean curvature $\vec{H}$ and Gauss curvature $K$, the Willmore functional can be defined as \[W(x)=\int_M (|\vec{H}|^2-K+1)dM.\]
Here $1$ is the curvature of $S^n$. Since Willmore \cite{Willmore1965} introduced this functional and his famous conjecture, it has played an important role in the study of global differential geometry and many new insights and methods have been developed for the proof of Willmore conjecture, see for example \cite{Li-y, Marques, Marques2} and reference therein.
 The Willmore conjecture states that it attains the minimum of Willmore functional among all tori in $S^3$, which has been shown to be true in \cite{Marques, Marques2} recently. As a consequence, this shows that Clifford torus is a stable Willmore torus in $S^3$.

 Examples of Willmore surfaces, the critical surfaces of Willmore functional, provide also many interesting objects in geometry.
 Minimal surfaces in space forms give important special examples of  Willmore surfaces \cite{Weiner, Ejiri1988}. In \cite{Ejiri1982}, Ejiri gave the first example of Willmore surface in $S^5$, which is not conformally equivalent to any minimal surface in any space forms. Later, in \cite{Pinkall1985} Pinkall provided the first examples of Willmore tori in $S^3$ which are different from minimal surfaces in space forms by the famous Hopf map and elastic curves.  Using  tensor product of curves in spheres, in \cite{Li-V}, Li and Vrancken derived many new examples of Willmore tori, including Ejiri's example as a special case.

In this paper, we show that the Ejiri torus can be viewed as the torus having lowest Willmore functional among all Willmore tori of tensor product.  To be concrete, by a simple observation, we prove that for all Willmore tori of tensor product, the Willmore functional $ W(y)\geq2\pi^2\sqrt{3}$, with the equality holding if and only the torus is the Ejiri torus (Theorem \ref{th-main}). This result also indicates that Ejiri torus is stable in the set of tori of tensor product in $S^5$.

 The Clifford torus is the first and the only known stable Willmore torus in $S^3$ \cite{Weiner} (See also  \cite{GLW}, \cite{Palmer}).
It is an interesting question that whether Clifford torus is the only stable Willmore torus in $S^n$ or not.
Note that if this is true, then Willmore conjecture follows directly.
 The un-stability of the Ejiri torus shown in this paper supports this conjecture partially. To be concrete, we construct a family of tori in $S^7$ containing the Ejiri torus, showing that the Ejiri torus is unstable. Similarly, we show that all Willmore tori derived by tensor product are unstable, when the co-dimension is big enough.
 Moreover, we also find a family of homogeneous tori in $S^5$ which contains the Ejiri torus and the Ejiri torus attains the maximal Willmore energy among them, which also shows the  un-stability of the Ejiri torus.

 Noticing that the Willmore functional of Ejiri  torus is very small, it is natural have a conjecture as below, which is true for Willmore tori of tensor product by Theorem \ref{th-main}:\vspace{5mm}\\
 { \bf Conjecture.} {\em Let $y$ be a Willmore torus $y$ in $S^n$, $n\geq 5$. If the Willmore functional $W(y)$ of $y$ satisfies
 \begin{equation}
 2\pi^2\leq W(y) \leq 2\pi^2\sqrt{3},
\end{equation}
then either $y$ is conformally equivalent to the Clifford torus with $W(y)=2\pi^2$, or  $y$ is conformally equivalent to the Ejiri torus with $W(y)= 2\pi^2\sqrt{3}$.\vspace{5mm}}

Note that $2\pi^2\sqrt{3}\approx 10.88 \pi$ is between $10\pi$ and $12\pi$. So one can derive by careful discussions that it is not possible to derive minimal torus in $\mathbb{R}^n$ with planer ends and with Willmore functional low than  $2\pi^2\sqrt{3}$. In fact for a minimal torus in $\mathbb{R}^n$ with planer ends to have total curvature $\geq-12\pi$, it has to have at most two ends. Then it has to be located in some $\mathbb{R}^4$ and has two ends. It is not possible to write down a Weierstrass representation via elliptic functions on a torus (See for example \cite{Bryant1984} and \cite{Costa} for similar discussions).

 We remark also that in \cite{GL}, they constructed a $\mathbb{CP}^3-$family Willmore tori in $S^4$ with Willmore functional $2\pi^2 n$. Here $n$ is the largest number of the Pythagorean triples $(p, q, n)\in \mathbb{Z}^3$.

The value distribution of Willmore two-spheres in $S^4$ has been proved to be $4\pi n$ in \cite{Bryant1984, Montiel}, $n\in \mathbb{Z}^+$. While the study of value distribution of Willmore tori are still very few. We hope this will induce more works on this direction.

We notice that in the construction of variations for the unstability of Willmore tori of tensor product, one can keep the conformal structure invariant. So such surfaces are in fact unstable under the variation with constraint. It is therefore natural to consider constrained Willmore surfaces of tensor product. For such surfaces, we obtain the similar results as Willmore tori of tensor product.

The paper is organized as follows. We will first recall the basic facts of surfaces in $S^n$ in Section 2 and then provide the proof in Section 3. Then we classify all constrained Willmore surfaces of tensor product in Section 4, which generalizes the result of Li and Vrancken \cite{Li-V}. Section 5 ends this paper by showing the un-stability of the Ejiri torus in $S^5$.

\section{Willmore surfaces in $S^n$}

The projective lightlike cone model of $S^n$ maps a point $x\in S^n$ into the projective lightlike cone $Q^n\subset \mathbb{R}P^{n+1}$ with
$Q^n=\{[X]|X=(1,x)\in \mathbb R^{n+2}, x\in S^n\}.$
Here $X$ is contained in the light cone $\mathcal{C}^{n+1}$ of $\mathbb R^{n+2}$ equipped with a Lorenzian metric
$\langle Y,Y\rangle=-y_{0}^2+y_1^2+\cdots+y_n^2$ for $Y=(y_0,y_1,\cdots,y_n)$, which we will denote by $\mathbb R^{n+2}_1$.

Let $y:M\rightarrow S^n$ be a conformal surface with $M$ a Riemann surface. Let $U\subset M$ be an open subset with complex coordinate $z$. Then  $< y_{z},y_{z}>=0 \hbox{ and } <
y_{z},y_{\bar{z}}> =\frac{1}{2}e^{2\omega}.$ Here $<\cdot,\cdot>$ is the standard inner product.
Then one has a natural lift of $y$ into $\mathbb R^{n+2}_1$ by
\[Y=e^{-\omega}(1,y): U\rightarrow \mathcal{C}^{n+1} \hbox{ with } |{\rm d}Y|^2=|{\rm d}z|^2,\] called a canonical lift with respect
to $z$. Moreover, there is a natural decomposition $M\times
\mathbb{R}^{n+2}_{1}=V\oplus V^{\perp}$, where
\begin{equation}
V={\rm Span}\{Y, Re(Y_z), Im(Y_z),Y_{z\bar{z}}\}
\end{equation}
is a Lorentzian rank-4 subbundle independent to the choice of $Y$
and $z$. Denoted by $V_{\mathbb{C}}=V\otimes \mathbb C$ and
$V^{\perp}_{\mathbb{C}}=V^{\perp}\otimes \mathbb C$. Let $\{Y,Y_{z},Y_{\bar{z}},N\}$ be a frame of
$V_{\mathbb{C}}$, with $N\in\Gamma(V)$ uniquely determined by
\begin{equation}\label{eq-N}
\langle N,Y_{z}\rangle=\langle N,Y_{\bar{z}}\rangle=\langle
N,N\rangle=0,\langle N,Y\rangle=-1.
\end{equation}
Since $Y_{zz}$ is orthogonal to
$Y$, $Y_{z}$ and $Y_{\bar{z}}$, there exists a local complex function $s$
and a local section $\kappa\in \Gamma(V_{\mathbb{C}}^{\perp})$ such that
\[Y_{zz}=-\frac{c}{2}Y+\kappa.\]
This defines two basic invariants $\kappa$ and $c$, \emph{the conformal Hopf differential} and
\emph{the Schwarzian} of $y$ (for more details, see \cite{BPP}). Let $D$ denote normal connection of $V_{\mathbb{C}}^{\perp}$ and $\psi\in
\Gamma(V_{\mathbb{C}}^{\perp})$  be a section of $V_{\mathbb{C}}^{\perp}$,
the structure equations are as follows:
\begin{equation}\label{eq-moving}
\left\{\begin {array}{lllll}
Y_{zz}&=&-\frac{c}{2}Y+\kappa,  \\
Y_{z\bar{z}}&=&-\langle \kappa,\bar\kappa\rangle Y+\frac{1}{2}N,\\
N_{z}&=&-2\langle \kappa,\bar\kappa\rangle Y_{z}-sY_{\bar{z}}+2D_{\bar{z}}\kappa,\\
\psi_{z}&=&D_{z}\psi-2\langle \psi,D_{\bar{z}}\kappa\rangle Y-2\langle
\psi,\kappa\rangle Y_{\bar{z}},
\end {array}\right.
\end{equation}
with the following integrable
conditions (the conformal Gauss, Codazzi and Ricci equations):
\begin{subequations}  \label{eq-integ:1}
\begin{align}
 & c_{\bar{z}}=6\langle
\kappa,D_z\bar\kappa\rangle +2\langle D_z\kappa,\bar\kappa\rangle,        \label{eq-integ:1A} \\
  &{\rm Im}\left(D_{\bar{z}}D_{\bar{z}}\kappa+\frac{\bar{c}}{2}\kappa\right)=0,    \label{eq-integ:1B} \\
 &  R^{D}_{\bar{z}z}=D_{\bar{z}}D_{z}\psi-D_{z}D_{\bar{z}}\psi =
2\langle \psi,\kappa\rangle\bar{\kappa}- 2\langle
\psi,\bar{\kappa}\rangle\kappa.  \label{eq-integ:1C}
\end{align}
\end{subequations}
 We define the Willmore
functional and Willmore surfaces as below:
\begin{definition} \emph{The Willmore functional} of $y:M\rightarrow S^n$ is
defined as the area of M with respect to the metric above:
\begin{equation}
W(y):=2i\int_{M}\langle \kappa,\bar{\kappa}\rangle dz\wedge
d\bar{z}.
\end{equation}
We call $y$ a \emph{Willmore surface}, if it is a critical surface
of the Willmore functional with respect to any variation of the map
$y:M\rightarrow S^n$.
\end{definition}
Denote by  $\vec{H}$ and $K$ the mean curvature and Gauss curvature of $y$ in $S^n$. Then it is direct to verify that $W(y)=\int_M (|\vec{H}|^2-K+1)dM$, coinciding with the original definition of Willmore functional.

Willmore surfaces
can be characterized as
\cite{Bryant1984,BPP,Ejiri1988}
\begin{theorem}\label{thm-willmore} $y$ is a Willmore surface if and only if the conformal Hopf differential $\kappa$ of $y$ satisfies the
Willmore condition
\begin{equation}\label{eq-willmore}
D_{\bar{z}}D_{\bar{z}}\kappa+\frac{\bar{c}}{2}\kappa=0.
\end{equation}
\end{theorem}
 \section{Surfaces of tensor product and the Ejiri torus}
In \cite{Li-V}, a full description of Willmore surfaces of tensor product has been provided by the moving frame methods of surfaces in $S^n$ and the discussions of solution to special elastic curve equation. In this section, we will first recall the properties of surfaces of tensor product and then characterize the Ejiri torus in \cite{Ejiri1982} as the Willmore torus with lowest Willmore functional $2\pi^2\sqrt{3}$ among all tori of tensor product.

 \subsection{Surfaces of tensor product}

To begin with, let us collect some basic descriptions of surfaces of tensor product.

Let $\gamma(s):S^1(L)\rightarrow S^n$ be a closed curve with Frenet
frame $\{\gamma,\beta_i,0\leq i\leq n-1\}$ as follows:
\begin{equation}\gamma'=\beta_0,\ \beta_0'=k_1\beta_1-\gamma,\ \beta_1'=k_2\beta_2-k_1\beta_0,\ \cdots,\
\beta_{n-1}'=-k_{n-2}\beta_{n-2},\end{equation}
where $\langle\gamma,\beta_i\rangle=0,\langle\beta_i,\beta_j\rangle=\delta_{ij}.$

Let $\hat\gamma(\hat{s}):S^1(\hat{L})\rightarrow S^m$ be a closed curve with Frenet
frame $\{\hat\gamma,\hat\beta_i,0\leq i\leq m-1\}$ as follows:
\begin{equation}\hat\gamma'=\beta_0,\ \hat\beta_0'=\hat{k}_1\hat\beta_1-\hat\gamma,\ \hat\beta_1'=\hat{k}_2\hat\beta_2-\hat{k}_1\hat\beta_0,\ \cdots,\
\hat\beta_{m-1}'=-\hat{k}_{m-2}\hat\beta_{m-2},\end{equation}
where $\langle\hat\gamma,\hat\beta_i\rangle=0,\langle\hat\beta_i,\hat\beta_j\rangle=\delta_{ij}.$

Recall that the tensor product of $x\in \mathbb{R}^n$ and $\hat{x}\in \mathbb{R}^m$ is of the form (\cite{Chen}, \cite{Li-V})
\begin{equation}x\otimes\hat{x}:=(x_1\hat{x}_1,\cdots,x_1\hat{x}_m,x_2\hat{x}_1,\cdots,x_2\hat{x}_m,\cdots,x_n\hat{x}_1,\cdots,
x_n\hat{x}_m)\in \mathbb{R}^{nm}.
\end{equation}
 It is direct to derive this lemma (\cite{Li-V})
 \begin{lemma}\cite{Li-V}
 $\langle f\otimes\hat{f},g\otimes\hat{g}\rangle=\langle f, g\rangle\langle\hat{f},
 \hat{g}\rangle,$ $\forall f,g\in \mathbb{R}^{n},\ \forall \hat{f},\ \hat{g}\in \mathbb{R}^{m}$.
 \end{lemma}
So we obtain a torus via the tensor product
 \[y:=\gamma(s)\otimes\hat\gamma(\hat s):T^2\rightarrow S^{(n+1)(m+1)-1}.\]
 We also have that
 \[Y=\left(1,\gamma(s)\otimes\hat\gamma(\hat {s})\right)\]
 is a canonical lift of $y$ into $\mathbb{R}^{(n+1)(m+1)+1}_{1}$ with respect to the complex coordinate $z=s+i\hat{s}$. So
\begin{equation}\label{eq-yzz}
\left\{\begin{array}{llll}
Y_z&=&\frac{1}{2}(0,\beta_0\otimes\hat\gamma)-\frac{i}{2}(0,\gamma\otimes\hat\beta_0),\\
Y_{z\bar{z}}&=&\frac{1}{4}k_1(0,\beta_1\otimes\hat\gamma)
+\frac{1}{4}\hat{k}_1(0,\gamma\otimes\hat\beta_1)
-\frac{1}{2}(0,\gamma\otimes\hat\gamma),\\
Y_{zz}&=& \frac{1}{4}k_1(0,\beta_1\otimes\hat\gamma)
-\frac{1}{4}\hat{k}_1(0,\gamma\otimes\hat\beta_1)
-\frac{i}{2}(0,\beta_0\otimes\hat\beta_0).
\end{array}\right.
\end{equation}
And from the structure equations we have that
\begin{equation}\label{eq-c-wp}
c=4\langle Y_{zz},Y_{z\bar{z}}\rangle =\frac{1}{4}(k_1^2-\hat k_1^2),\end{equation}
and
\begin{equation}\label{eq-kappa-wp}\langle\kappa,\bar\kappa\rangle=\langle Y_{zz},Y_{\bar{z}\bar{z}}\rangle=\langle Y_{z\bar{z}},Y_{z\bar{z}}\rangle=\frac{1}{4}+\frac{1}{16}(k_1^2+\hat{k}_1^{2}). \end{equation}

We need the following lemma, which can be found in the discussions above Theorem 3.1 on page 146 of \cite{Langer-Singer-agag}(See also \cite{Langer-Singer1984, Langer-Singer-BLMS} for more discussions). For the reader's convenience, we give a proof as below.
\begin{lemma} \label{lemma} \cite{Langer-Singer-agag}
 Let $\gamma(s):[0,L]\rightarrow S^n$ be a $C^2$ closed curve with length $L$ and curvature $k$ in $S^n $. Let $a_0\geq 2$ be a constant. Then
\begin{equation}
\mathcal{E}(\gamma):=\int_0^L(k^2+a_0)ds\geq 4\pi\sqrt{a_0-1}.
\end{equation}
And the equality holds if and only if $\gamma$ is a circle in $S^2 $ with curvature $\sqrt{a_0-2}$.
\end{lemma}
\begin{proof} It is direct to see
\[\mathcal{E}(\gamma)=\int_0^L(k^2+a_0)ds\geq 2\int_0^L\sqrt{k^2+1}\sqrt{a_0-1}ds.\]
Since $\sqrt{k^2+1}$ is the curvature of $\gamma$ in $\mathbb{R}^{n+1}$, by Fenchel Theorem, \[\int_0^L\sqrt{k^2+1}\sqrt{a_0-1}ds\geq2\pi \sqrt{a_0-1},\]
 and equality holds if and only if $\gamma$ is a convex closed curve contained in a 2-dimensional plane of $\mathbb{R}^{n+1}$.
For the equality case, $\gamma$ is the intersection of a 2-dimensional plane and $S^n(1)$, hence $k=constant$. Moreover, we also have
$k^2+1=a_0-1$, i.e., $k=\sqrt{a_0-2}$.
\end{proof}

\subsection{The Ejiri torus}
\begin{theorem}\label{th-main} Let $y=\gamma\otimes\hat\gamma$ be a Willmore torus of tensor product. Then
\[W(y)\geq2\pi^2\sqrt{3},\]
with the equality holding if and only if $y$ is the torus of Ejiri in \cite{Ejiri1982}:
\begin{equation}\label{eq-ejiri-t}
\sqrt{\frac{1}{3}}\left(\cos \hat s\cos \sqrt{3}s,\sin \hat s\cos \sqrt{3}s,\cos \hat s\sin \sqrt{3}s,\sin\hat s\sin \sqrt{3}s,\sqrt{2}\cos\hat s,\sqrt{2}\sin \hat s\right).\end{equation}
\end{theorem}
\begin{proof}
Since $y$ is a Willmore torus of tensor product, by Theorem 1 of \cite{Li-V}, one of $\gamma$ and $\hat{\gamma}$ is a big circle in a sphere. Without loss of generality, we assume that $\hat{k}_1=0$. By \eqref{eq-kappa-wp}, we have that
\[W(y)=4\int_M\langle\kappa,\bar\kappa\rangle ds d\hat{s}=4\int_0^{2\pi}d\hat{s}\int_{\gamma}\left(\frac{1}{4}+\frac{1}{16}k_1^2\right)ds= \frac{\pi}{2}\int_{\gamma}(k_1^2+4)ds.\]
By Lemma \ref{lemma}, or the discussions above Theorem 3.1 in page 146 of \cite{Langer-Singer-agag}, \[\int_{\gamma}(k_1^2+4)dt\geq 4\pi\sqrt{3}\] with equality holding if and only if $k_1=\sqrt{2}, $ $k_2\equiv0$. Elementary computation shows that this gives the Ejiri torus \eqref{eq-ejiri-t}.
\end{proof}
From the proof we obtain that
\begin{corollary} Let $y=\gamma\otimes\hat\gamma$ be a torus of tensor product with $\hat\gamma(\hat{s})=(\cos\hat s, \sin\hat s)$. Then
\[W(y)\geq2\pi^2\sqrt{3},\]
with the equality holding if and only if $y$ is the Ejiri's torus \eqref{eq-ejiri-t}.
\end{corollary}

\begin{remark}Note that when $k_1\equiv0$, the image of $y$ is exactly a double cover of the classical Clifford torus, and hence its Willmore functional $W(y)=4\pi^2$. Below we see a family tori of tensor product, with the double cover of the classical Clifford torus as the one with highest Willmore functional and the limiting surface provides the classical Clifford torus. We also remark that for any torus $y$ of tensor product, $W(y)>2\pi^2$ (Proposition \ref{prop-wy}).
\end{remark}

\begin{example}\label{example-inf} Let $0<a\leq1$ and $a^2+b^2=1, \ b\geq0$. Set
\[y_a=\gamma_a(s)\otimes \hat\gamma_a(\hat s), \hbox{ with } \gamma_a(s)=\left(a\cos  \frac{s}{a}, a \sin  \frac{s}{a} , b\right),\ \hat\gamma_a(\hat s)=\left(a\cos  \frac{\hat s}{a}, a \sin  \frac{\hat s}{a} , b\right).\  \]
We see that for all $a\in(0,1]$, $y_a$ is a flat torus in some $S^7(a)\subset S^8$. Moreover, when $a=1$, $y_a$ is a double cover of the classical Clifford torus in some $S^3(1)\subset S^8$, and when $a\rightarrow 0$, by some scaling, $y_a$ tends to the clifford torus $(\cos s, \sin s, \cos\hat s, \sin \hat s)$.
Concerning the Willmore functional of $y_a$, by \eqref{eq-kappa-wp}, we have that
\begin{equation*}
\begin{split}W(y_a)=4\int_{\hat\gamma_a}\int_{\gamma_a}\left(\frac{1}{4}+\frac{1}{16}k_1^2+\frac{1}{16}\hat k_1^2\right)dsd\hat s = 2\pi^2(1+a^2)>2\pi^2.
\end{split}\end{equation*}

\end{example}

\begin{example}
From Theorem \ref{th-main}, it is natural to conjecture that Ejiri's torus is a stable Willmore torus in $S^n$. However, the following examples show that this is not true when $n\geq7$.

 Let $0<a\leq1$ and $a^2+b^2=1, \ b\geq0$. Set
\[\tilde{y}_a=\gamma (s)\otimes \hat\gamma_a(\hat s)\] with \[ \gamma (s)=\left(\sqrt{\frac{ 1 }{ 3}}\cos   \sqrt{3}s, \sqrt{\frac{ 1 }{ 3}}\sin   \sqrt{3}s, \sqrt{\frac{ 2 }{ 3}} \right),\ \hat\gamma_a(\hat s)=\left(a\cos  \frac{\hat s}{a}, a \sin  \frac{\hat s}{a} , b\right).\  \]
We see that for all $a\in(0,1]$, $\tilde{y}_a$ is a flat torus in some $S^7(r)\subset S^8$, $r=\sqrt{1-\frac{2b^2}{3}}$. Moreover, when $a=1$, $\tilde{y}_a$ is the Ejiri's torus.
Concerning the Willmore functional of $\tilde{y}_a$, by \eqref{eq-kappa-wp}, we have that
\begin{equation*}
\begin{split}W(\tilde{y}_a)&=4\int_{\hat\gamma_a}d\hat s\int_{\gamma}\left(\frac{1}{4}+\frac{1}{16}k_1^2+\frac{1}{16}\hat k_1^2\right)ds= \frac{\pi^2\left(5a^2+1\right)}{a\sqrt{3}}.
\end{split}\end{equation*}
So $W(\tilde{y}_a)$ is an increasing function in $a$ for $a\in\left[\frac{1}{\sqrt{5}},1\right]$, which means that Ejiri's torus is unstable in $S^7$.
\end{example}

\begin{proposition}
Let $y=\gamma\otimes\hat\gamma:S^1(1)\times S^1(\sqrt{\frac{1}{3}})\rightarrow S^5\subset S^7$ be the Ejiri's torus in \eqref{eq-ejiri-t}. Then $y$ is unstable in $S^7$.
\end{proposition}

Later in Section 5 we will see that the Ejiri torus is unstable in $S^5$.
\subsection{Unstability of torus of tensor product}

The one parameter family of tori deriving the unstability of the Ejiri torus can be deformed easily to show that all   tori of tensor product are unstable.
\begin{example}\label{example-unstable} Let $y_1=\gamma (s)\otimes \hat\gamma (\hat s)$ be a torus of tensor product as before.  Let $0<a\leq1$ and $a^2+b^2=1, \ b\geq0$. Set
\[y_a=\gamma_a(\theta)\otimes \hat\gamma_a(\hat \theta),\ \hbox{ with }\ \gamma_a(\theta)=\left(a\gamma\left(\frac{\theta}{a}\right), b\right), \ \hat\gamma_a(\hat \theta)=\left(a\gamma\left(\frac{\hat \theta}{a}\right), b\right).\]
Then it is direct to derive that the curvature $k_{1,a}(\theta)$ of $\gamma_a(\theta)$ is $\frac{1}{a}k_1(\frac{\theta}{a})$ and the curvature $\hat k_{1,a}(\theta)$ of $\hat\gamma_a(\hat \theta)$ is $\frac{1}{a}\hat k_1(\frac{\hat \theta}{a})$.
By \eqref{eq-kappa-wp} and Lemma \ref{lemma}, we have that
\begin{equation*}
\begin{split}
W(y_a)&=\frac{1}{4}\iint_{\gamma_a\otimes\hat\gamma_a}\left(4+\frac{1}{a^2}k_1^2\left(\frac{\theta}{a}\right)+ \frac{1}{a^2}\hat k_1^2\left(\frac{\hat \theta}{a}\right)\right)d\theta d\hat{\theta}\\
&=\frac{1}{4}\iint_{\gamma\otimes\hat\gamma}\left(4a^2+ k_1^2(s)+  \hat k_1^2(\hat s)\right)dsd\hat{s}\\
&\leq \frac{1}{4}\iint_{\gamma\otimes\hat\gamma}\left(4+ k_1^2(s)+ \hat k_1^2(\hat s)\right)ds d\hat{s}=W(y_1).\\
\end{split}\end{equation*}
\end{example}
So ignoring the co-dimensional restriction, we obtain the following
\begin{theorem}\label{th-main-3} Let $y=\gamma\otimes\hat\gamma$ be a (Willmore) torus of tensor product. Then it is unstable.
\end{theorem}

\begin{remark} For Willmore tori of tensor product, we also have a simpler family to show the un-stability of them. Let $0<a\leq1$ and $a^2+b^2=1, \ b\geq0$. Let $y_1=\gamma (s)\otimes \hat\gamma (\hat s)$ be a Willmore torus of tensor product. So we have $\hat\gamma=(\cos \hat s,\sin \hat s)$.  Let $0<a\leq1$ and $a^2+b^2=1, \ b\geq0$. Set
\[y_a=\gamma_a(\theta)\otimes \hat\gamma(\hat \theta),\ \hbox{ with }\ \gamma_a(\theta)=\left(a\gamma\left(\frac{\theta}{a}\right), b\right).\]
So we have that
\begin{equation*}
\begin{split}
W(y_a)=\frac{1}{4}\iint_{\gamma_a\otimes\hat\gamma}\left(4+\frac{1}{a^2}k_1^2\left(\frac{\theta}{a}\right)\right)d\theta d\hat{s}=\frac{1}{4}\iint_{\gamma\otimes\hat\gamma}\left(4a^2+ k_1^2(s)\right)dsd\hat{s}.\\
\end{split}\end{equation*}
So we have $W(y_a)\leq W(y)$ for all $a\in(0,1]$ and $W(y_a)\leq W(y)$ if and only if $a=1$. This indicates that all Willmore tori of tensor product are unstable in $S^9$ since by Theorem 1 of \cite{Li-V}, such surfaces must be in some $S^7$.
\end{remark}

\vspace{2mm}

We end this section by a bit more discussion on the values of Willmore functional for tori of tensor product. From Example \ref{example-inf}, we see that there exists a family tori of tensor product with their Willmore functional tending to $2\pi^2$. The following proposition shows that $2\pi^2$ is exactly the infimum of the Willmore functional for tori of tensor product. Note that it is also a corollary of a theorem of Chen \cite{Chen2}, which states that for a closed flat surface $x$ in $\R^n$, its Willmore functional $W(x)\geq2\pi^2$, with equality holding when $x$ is the standard Clifford torus (See also \cite{Chen-book}, \cite{Hou} for a proof).
\begin{proposition} \label{prop-wy}
Let $y=\gamma\otimes\hat\gamma$ be a torus of tensor product. Then $W(y)> 2\pi^2$.
\end{proposition}

\begin{proof}
By \eqref{eq-kappa-wp} and Fenchel Theorem, we have that
\begin{equation*}
\begin{split}
W(y)&=\frac{1}{4}\iint_{\gamma\otimes\hat\gamma}\left(4+k_1^2+\hat{k}_1^2\right)dsd\hat{s}\\
&>\frac{1}{4}\iint_{\gamma\otimes\hat\gamma}\left(2+k_1^2+\hat{k}_1^2\right)dsd\hat{s}\\
&\geq \frac{1}{2}\int_{\hat\gamma}\sqrt{\hat k_1^2+1}d\hat{s} \int_{\gamma} \sqrt{k^1+1}ds\\
&\geq 2\pi^2. \\
\end{split}\end{equation*}

\end{proof}

\section{Constrained Willmore surfaces of tensor product}

 This section provides all constrained Willmore surfaces of tensor product, which is a generalization of Theorem 1.1 of \cite{Li-V}.
 Applying to Willmore surfaces, we re-obtain Theorem 1.1 of \cite{Li-V}, from the version of conformal geometry, which also indicates that the Ejiri torus is a Willmore torus \cite{Ejiri1982}.

We recall that an immersion from a Riemann surface $M$ into $S^n$ is called a {\em constrained Willmore surface } if it is critical surface of Willmore functional under all variations fixing the conformal structure of $M$ \cite{BoPP}, \cite{BPP}. It is known that a surface $y$ is constrained Willmore if and only if there exists some holomorphic function $q$ such that the Hopf differential $\kappa$ and the Schwarzian $c$ satisfy the following equation (\cite{BoPP}, \cite{BPP})
\begin{equation}\label{eq-CWillmore}
  D_{\bar{z}}D_{\bar{z}}\kappa+\frac{\bar c}{2}\kappa=Re(q\kappa).
\end{equation}

We retain the notations in Section 2 for a torus $y$ of tensor product. To begin with,
set
\begin{equation}\left\{
\begin{split}
&E_j=(0,\beta_j\otimes\hat\gamma),j=2,\cdots,n-1;\\
&F_j=(0,\gamma\otimes\hat\beta_j),j=2,\cdots,m-1;\\
&G_{jk}=(0,\beta_j\otimes\hat\beta_k),j=0,\cdots,n,\ k=0,\cdots,m;\\
&L_1=(0,\beta_1\otimes\hat\gamma)+\frac{k_1}{2}(1,\gamma\otimes\hat\gamma);\\
&L_2=(0,\gamma\otimes\hat\beta_1)+\frac{\hat{k}_1}{2}(1,\gamma\otimes\hat\gamma).\\
\end{split}\right.
\end{equation}
It is easy to see that they provide an orthogonal basis of the conformal normal bundle of $Y$. By this framing, we have that from \eqref{eq-yzz}
\begin{equation}\label{eq-kappa-wp}
\kappa=\frac{k_1}{4}L_1-\frac{\hat{k}_1}{4}L_2-\frac{i}{2}G_{00}.\end{equation}
Direct computations show that
\begin{equation}\label{eq-dz}\left\{
\begin{split}
&D_{\bar{z}}L_1=\frac{k_2}{2}E_2+\frac{i}{2}G_{10},\\
&D_{\bar{z}}L_2=\frac{i\hat{k}_2}{2}F_2+\frac{1}{2}G_{01},\\
&D_{\bar{z}}G_{00}=\frac{k_1}{2}G_{10}+\frac{i\hat{k}_1}{2}G_{01},\\
&D_{\bar{z}}G_{10}=-\frac{k_1}{2}G_{00}+\frac{k_2}{2}G_{20}-\frac{i}{2}L_{1}+\frac{i\hat{k}_1}{2}G_{11},\\
&D_{\bar{z}}G_{01}=-\frac{i\hat{k}_1}{2}G_{00}+\frac{i\hat{k}_2}{2}G_{02}-\frac{1}{2}L_{2}+\frac{k_1}{2}G_{11},\\
&D_{\bar{z}}E_2=\frac{k_3}{2}E_3-\frac{k_2}{2}L_1+\frac{i}{2}G_{20},\\
&D_{\bar{z}}F_2=\frac{\hat{k}_3}{2}F_3-\frac{i\hat{k}_2}{2}L_1+\frac{1}{2}G_{02}.\\
\end{split}\right.
\end{equation}
So
\[
\begin{split}
D_{\bar{z}}\kappa&=\frac{k_{1s}}{8}L_1-\frac{i\hat{k}_{1\bar s}}{8}L_2+\frac{k_1}{4}D_{\bar z}L_1-\frac{\hat{k}_1}{4}D_{\bar z}L_2-\frac{i}{2}D_{\bar z}G_{00}\\
&=\frac{k_{1s}}{8}L_1-\frac{i\hat{k}_{1\bar s}}{8}L_2+\frac{k_1k_2}{8} E_2+\frac{ik_1}{8}G_{10}-\frac{i\hat{k}_1\hat{k}_2}{ 8} F_2-\frac{\hat{k}_1}{8}G_{01}-\frac{ik_1}{4}G_{10}+\frac{\hat{k}_1}{4}G_{01}\\
&=\frac{k_{1s}}{8}L_1-\frac{i\hat{k}_{1\bar s}}{8}L_2+\frac{k_1k_2}{8} E_2-\frac{ik_1}{8}G_{10}-\frac{i\hat{k}_1\hat{k}_2}{ 8} F_2+\frac{\hat{k}_1}{8}G_{01}.\\
\end{split}\]
From this and \eqref{eq-dz} we see that
\[\langle D_{\bar z}D_{\bar{z}}\kappa, G_{11}\rangle=\langle -\frac{ik_1}{8}D_{\bar z}G_{10}+\frac{\hat{k}_1}{8}D_{\bar{z}}G_{01}, G_{11}\rangle=\frac{k_1\hat{k}_1}{8}.\]
Since
$\langle \kappa,G_{11}\rangle=0,$ for any holomorphic function $q$,
we have
\[\langle D_{\bar z}D_{\bar{z}}\kappa+\frac{\bar{c}}{2}\kappa-Re(q\kappa), G_{11}\rangle=\frac{k_1\hat{k}_1}{8}.\]
The constrained Willmore equation $D_{\bar{z}\bar{z}}+\frac{\bar{c}}{2}\kappa=Re(q \kappa)$  then forces
\[k_1\hat{k}_1=0.\]
Without lose of generality, assume that $\hat{k}_1=0$, that is, $\hat{\gamma}$ is a great circle.  Now we obtain
\[c=\frac{k_1^2}{4},\ \kappa=\frac{k_1}{4}L_1-\frac{i}{2}G_{00},\ D_{\bar{z}}\kappa=\frac{k_{1s}}{8}L_1+\frac{k_1k_2}{8} E_2-\frac{ik_1}{8}G_{10}.
\]
So
\begin{equation}\label{eq-dzdzk}
\begin{split}
16D_{\bar{z}}D_{\bar{z}}\kappa&= k_{1ss}L_1+(k_{1}k_2)_{s} E_2-ik_{1s}G_{10}+ 2k_{1s}D_{\bar{z}}L_1+2k_1k_2D_{\bar{z}}E_2-2ik_1D_{\bar{z}}G_{10}  \\
&=  k_{1ss}L_1+(2k_{1s}k_2+k_1k_{2s}) E_2 +k_1k_2(k_3E_3-k_2L_1)+ik_1(k_1G_{00}+iL_{1})\\
&= (k_{1ss}-k_1k_2^2-k_1)L_1+(2k_{1s}k_2+k_1k_{2s}) E_2 +k_1k_2k_3E_3+ik_1^2G_{00}.\\
\end{split}\end{equation}
 Then the constrained Willmore equation $D_{\bar{z}\bar{z}}+\frac{\bar{c}}{2}\kappa=Re(q \kappa)$ now reads
\begin{equation}\left\{
\begin{split}
&k_1''-k_1k_2^2-k_1+\frac{k_1^3}{2}=Re(\frac{qk_1}{4}),\\
&2k_1'k_2+k_1k_2'=0,\\
&0=Re(iq),\\
&k_1k_2k_3=0.\\
\end{split}\right.
\end{equation}
Since $q$ is holomorphic and $Im(q)=0$, $q$ is a constant real number. Set $q=q_1\in \R$. We obtain
\begin{equation}\left\{
\begin{split}
&k_1''-k_1k_2^2-k_1+\frac{k_1^3}{2}=\frac{q_1k_1}{4},\\
&2k_1'k_2+k_1k_2'=0,\\
&k_1k_2k_3=0.\\
\end{split}\right.
\end{equation}
In a sum, we have proved the following theorem.
\begin{theorem} The tensor product surface $y=\gamma\otimes\hat\gamma$ is a constrained Willmore surface if and only if one of $\gamma$ and $\hat\gamma$ is the great circle $S^1$, say, for example, $
\hat\gamma=S^1$, and the other one, say, $\gamma$, is a curve in some $S^3(1)$ satisfying
the following equations
\begin{equation}\label{eq-elastic2}\left\{
\begin{split}
&k_1''-k_1k_2^2-a_0k_1+\frac{k_1^3}{2}=0,\\
&2k_1'k_2+k_1k_2'=0.\\
\end{split}\right.
\end{equation}
Here $a_0\in \R$ is a constant. Moreover, $y$ is Willmore if and only if $a_0=1$.
\end{theorem}
Note that \eqref{eq-elastic2} is exactly the special case $\lambda=a_0$ and $G=0$ in (1.3) of \cite{Langer-Singer1984}. We refer to Section 2 of \cite{Langer-Singer1984} and Section 3-5 of \cite{Li-V} for more discussions on the solutions to \eqref{eq-elastic2} and corresponding constrained Willmore surfaces.  We also refer to \cite{Heller} for another kind of relationships between constrained Willmore surfaces and elastic curves.

It is easy to see that in this case $k_1$ and $k_2$ being constant yields all homogeneous constrained Willmore surfaces of tensor product. So we obtain
\begin{corollary}
Let $y=\gamma\otimes\hat\gamma$ be a  homogeneous constrained Willmore surface of tensor product.
Set  $\hat\gamma =     \left( \cos \hat s , \sin \hat s\right).$ Then $\gamma$ is of the form
\begin{equation}\label{eq-homo-c-W}
     \gamma =     \left(a\cos \frac{  s}{\sqrt{a^2+b^2\lambda^2}},a\sin  \frac{s}{\sqrt{a^2+b^2\lambda^2 }},b \cos  \frac{\lambda s}{\sqrt{a^2+b^2\lambda^2}}, b \sin \frac{\lambda  s}{\sqrt{a^2+b^2\lambda^2}}\right),
\end{equation}
with
\[k_1^2=\frac{a^2b^2(\lambda^2-1)^2 }{ (a^2+b^2\lambda^2)^2},~\hbox{
and }\
  k_2^2=\frac{ \lambda^2}{ (a^2+b^2\lambda^2)^2} \hbox { when } k_1\neq0. \]
\end{corollary}

This corollary indicates that not all homogeneous tori of tensor product are constrained Willmore surfaces.\\

Concerning the stability of constrained Willmore surface of tensor product, noticing in Example \ref{example-unstable} the conformal structure of tori are in fact invariant under the variation, we have the following
\begin{theorem}\label{th-main-2} Let $y=\gamma\otimes\hat\gamma$ be a constrained Willmore torus of tensor product. Then it is unstable.
\end{theorem}
\begin{remark}
Some related results on the stability of constrained Willmore surface can be found in \cite{KL, NS}.\\
\end{remark}

In the end, let us go back to Willmore surfaces. Setting $q_1=0$ in \eqref{eq-elastic2}, we re-obtain the theorem due to Li and Vrancken in \cite{Li-V}
\begin{theorem}\cite{Li-V} The tensor product surface $y=\gamma\otimes\hat\gamma$ is a Willmore surface if and only if one of $\gamma$ and $\hat\gamma$ is the great circle $S^1$, say, for example, $
\hat\gamma=S^1$, and the other one, say, $\gamma$, is a curve in some $S^3$ satisfying
the following equations
\begin{equation}\label{eq-elastic}\left\{
\begin{split}
&k_1''-k_1k_2^2-k_1+\frac{k_1^3}{2}=0,\\
&2k_1'k_2+k_1k_2'=0.\\
\end{split}\right.
\end{equation}
\end{theorem}

Note that the above equations are exactly the equations of free elastic curves in the hyperbolic space $H^{3}(-1)$ \cite{Langer-Singer1984, Langer-Singer-BLMS, Li-V}.\\

\section{Unstability of the Ejiri torus}

Let
\begin{equation}y_{\theta}(s,\hat{s})=\left(a_1\cos(s+\hat{s}), a_1\sin(s+\hat{s}), a_2\cos(s-\hat{s}), a_2\sin(s-\hat{s}),a_3\cos\hat{s}, a_3\sin\hat{s}\right)\end{equation}
with
\[a_1=\sqrt{\frac{1}{3}}\cos\theta,\ a_2=\sqrt{\frac{1}{3}}\sin\theta,\  \hbox{ and }  a_3=\sqrt{\frac{2}{3}}.\]
Then $y_{\theta}$, with  $\theta\in[0,\frac{\pi}{2}]$, provides one family of tori with $(s,\hat{s})\in[0,2\pi]\times[0,2\pi]$ in $S^5$. In particular,
   $y_{\theta}|_{\theta=\frac{\pi}{4}}$ gives the Ejiri torus (up to an isometry of $S^5$). We will show that the Ejiri torus is unstable by computation of the Willmore functional of $y_{\theta}$.

First we use the coordinate changing
\[\phi=\frac{1}{\sqrt{3}}(s+\hat{b}_1\hat{s}), \ \  \varphi=\frac{1}{b_3 }\hat{s},\]
such that
\[s+\hat{s}=\sqrt{3}\phi+b_1\varphi,\ s-\hat{s}=\sqrt{3}\phi-b_2\varphi. \ \]
Here \[ \ b_1= 2b_3\sin^2\theta , \ b_2= 2b_3\cos^2\theta, \ b_3=\sqrt{\frac{3}{2+\sin^2 2\theta}},\ \hbox { and } \hat{b}_1=1-\frac{b_1}{b_3}.\]
We compute now
\[y_{\theta, \phi}=\sqrt{3}\left(-a_1\sin(s+\hat{s}), a_1\cos(s+\hat{s}), -a_2\sin(s-\hat{s}), a_2\cos(s-\hat{s}),0, 0\right),\]
\[y_{\theta, \varphi}= \left(-a_1b_1\sin(s+\hat{s}), a_1b_1\cos(s+\hat{s}), a_2b_2\sin(s-\hat{s}), -a_2b_2\cos(s-\hat{s}),-a_3b_3\sin\hat{s}, a_3b_3\cos\hat{s}\right).\]
It is direct to see that $z=\phi+i\varphi$ is a complex coordinate of $y_{\theta}$ since
\[|y_{\theta, \phi}|^2=1,\ |y_{\theta, \varphi}|^2=a_1^2b_1^2+a_2^2b_2^2+a_3^3b_3^2=b_3^2\left(\frac{4\cos^2\theta\sin^2\theta}{3}+\frac{2}{3}\right)=1,\
\langle y_{\theta, \phi},y_{\theta, \varphi}\rangle=a_1^2b_1-a_2^2b_2=0.\]
We also have
\[y_{\theta, \phi\phi}=-3\left(a_1\cos(s+\hat{s}), a_1\sin(s+\hat{s}), a_2\cos(s-\hat{s}), a_2\sin(s-\hat{s}),0, 0\right),\]
\[y_{\theta, \varphi\varphi}=-\left(a_1b_1^2\cos(s+\hat{s}), a_1b_1^2\sin(s+\hat{s}), a_2b_2^2\cos(s-\hat{s}), a_2b_2^2\sin(s-\hat{s}),a_3b_3^2\cos\hat{s}, a_3b_3^2\sin\hat{s}\right).\]
So
\[|y_{\theta, \phi\phi}|^2=3,\ |y_{\theta, \varphi\varphi}|^2=a_1^2b_1^4+a_2^2b_2^4+a_3^2b_3^4=\frac{b_3^4}{3}(16(\cos^2\theta\sin^8\theta+\cos^8\theta\sin^2\theta)+2),\]
and
\[ \langle y_{\theta, \phi\phi},y_{\theta, \varphi \varphi}\rangle=3a_1^2b_1^2+3a_2^2b_2^2=b_3^2\sin^2 2\theta.\]
Therefore we obtain
\[\begin{split}|\vec{H}|^2+1&=\frac{1}{4}\left(|y_{\theta, \phi\phi}|^2+ |y_{\theta, \varphi\varphi}|^2+2\langle y_{\theta, \phi\phi},y_{\theta, \varphi \varphi}\rangle\right)\\
&=\frac{1}{4}\left(
3+\frac{b_3^4}{3}(16(\cos^2\theta\sin^8\theta+\cos^8\theta\sin^2\theta)+2)+2b_3^2\sin^22\theta\right)\end{split}\]
We have now
\[\begin{split}W(y_{\theta})&=\int_M(|\vec{H}|^2+1)dM\\
&=\int_0^{2\pi}\int_0^{2\pi}(|\vec{H}|^2+1)\frac{1}{\sqrt{3}b_3}ds d\hat{s}\\
&=\frac{\pi^2\sqrt{2+\sin^22\theta}}{3}\left(
3+\frac{b_3^4}{3}(2+16(\cos^2\theta\sin^8\theta+\cos^8\theta\sin^2\theta))+2b_3^2\sin^2 2\theta\right)\\
&= \pi^2\sqrt{2+\sin^22\theta} \left(1+\frac{1}{(2+\sin^22\theta)^2}\left(2+4\sin^22\theta-3\sin^42\theta)\right)+\frac{2\sin^22\theta}{2+\sin^22\theta}\right)
\end{split}\]
Let $\rho=\sqrt{2+\sin^22\theta}$. So $\sin^22\theta=\rho^2-2$ and
\[\begin{split}W(y_{\theta})&= \pi^2\rho \left(1+\frac{1}{\rho^4}\left(2+4(\rho^2-2)-3(\rho^2-2)^2)\right)+\frac{2(\rho^2-2)}{\rho^2}\right)\\
&= \pi^2\frac{1}{\rho^3}\left(12\rho^2-18\right)\\
&= 6\pi^2\left(\frac{2}{\rho}-\frac{3}{\rho^3}\right).\\
\end{split}\]
Since $ \frac{2}{\rho}-\frac{3}{\rho^3}$ is an increasing function for $\rho\in[\sqrt{2},\sqrt{3}]$, the Ejiri torus attains maximal Willmore functional. Hence we obtain
\begin{proposition} The Ejiri torus is unstable in $S^5$.
\end{proposition}

\ \\

{\bf Acknowledgements}: The author is thankful to Professor Xiang Ma
for his many valuable discussions and suggestions. The author is
supported by  the Project 11201340 of NSFC  and
the Fundamental Research Funds for the Central Universities.

{\small
\def\refname{Reference}

}

\ \\\\
Peng Wang,\\
Department of Mathematics, \\
Tongji University, \\
Siping Road 1239, Shanghai, 200092,\\
 P. R. China.\\
{\em E-mail address}: {netwangpeng@tongji.edu.cn}

\end{document}